\documentclass{article}

\usepackage[utf8]{inputenc}
\usepackage{amsmath}
\usepackage{amssymb}

\usepackage{color}
\usepackage{array}
\usepackage{amsthm,url,graphicx}

\newcolumntype{x}{>{\centering\arraybackslash\hspace{0pt}}p{4.5mm}}
\newtheorem{thm}{Theorem}[section]
\newtheorem{cor}[thm]{Corollary}
\newtheorem{lem}[thm]{Lemma}

\title{Farey fractions with equal numerators and the rank of unit fractions}
\author{Rogelio Tom\'as Garc\'ia 
\thanks{rogelio.tomas@cern.ch}\\
{\centering \it CERN, Esplanade des Particules 1, 1211 Meyrin, Switzerland}}

\date{\today}

\begin{document}

\maketitle
\abstract{Analytical expressions are derived for the number of fractions with equal numerators in the Farey sequence of order $n$, $F_n$, and in the truncated Farey sequence $F_n^{1/k}$ containing all Farey fractions below  $1/k$, with $1\leq k \leq n$. 
These developments lead to an expression for the rank of $1/k$ in $F_n$, or equivalently $\left|F_n^{1/k}\right|$, and to remarkable relations between the ranks of different unit fractions. Furthermore, the results are extended to Farey fractions of the form $2/k$.}

\section{Introduction}

The Farey sequence $F_n$ of order $n\in\mathbb{N}$ is an ascending sequence of irreducible fractions between
0/1 and 1/1 whose denominators do not exceed $n$, see e.g.~\cite{hardy,kanemitsu,tomas,daniel,andrey}.
Throughout the paper we exclude the fraction $0/1$ from $F_n$.
We also define the sequence $F_n^{1/k}$ as
\[
F_n^{1/k} = \left\{ \alpha \in F_n : \alpha < 1/k \right\} \ , \ \ k>0 \ .
\]

The number of Farey fractions with  denominators equal to $d$ in $F_n$ is well known to be given by Euler's Totient function, $\varphi(d)$, when $d\leq n$. It is also well known that the sum of all denominators in $F_n$ is twice the sum of all numerators~\cite{blake}. However, expressions for the number of fractions with equal numerators in $F_n$ are not given in the literature. 
Here, we define $\mathcal{N}_n(h)$ as the number of fractions with numerators equal to $h$ in $F_n$. A closely related quantity is derived in Proposition 1.29 in~\cite{andrey}, that is the number of Farey fractions in $F_n$ with numerators below or equal to $m$ as
\begin{equation}\label{ansumN}
\sum_{h=1}^{m}  \mathcal{N}_n(h) = \frac{1}{2} + \sum_{d\geq 1} \mu(d)\left\lfloor\frac{m}{d}\right\rfloor\left( \left\lfloor\frac{n}{d}\right\rfloor - \frac{1}{2}\left\lfloor\frac{m}{d}\right\rfloor \right) \ ,
\end{equation}
note that we have removed $1$ from the expression in Proposition 1.29 in~\cite{andrey} as we exclude 0/1 from $F_n$.

$\mathcal{N}^{1/k}_n(h)$ is defined as the number of fractions with numerators equal to $h$ in $F_n^{1/k}$.
In Section~\ref{sec1}, analytical expressions for $\mathcal{N}_n(h)$ and $\mathcal{N}^{1/k}_n(h)$ are derived that allow to reveal some properties of $\mathcal{N}_n(h)$. The variation of $\mathcal{N}_n(h)$ versus $n$ and $h$ is illustrated by establishing remarkably simple 
identities involving $\mathcal{N}_n(2h)$ and $\mathcal{N}_{n+ph}(h)$ for any $p\geq 0$.

We define $I_n(1/k)=\left|F_n^{1/k}\right|$ as the rank of $1/k$ in $F_n$. In Section~\ref{sec2} new analytical expressions for $I_n(1/k)$ are developed using the results in~Section~\ref{sec1} for $\mathcal{N}^{1/k}_n(h)$.
These expressions could help in the development of efficient algorithms to compute the rank of Farey  fractions and the related order-statics problem~\cite{andrey,Jakub}. Furthermore, $I_n(1/k)$ appears when deriving estimates for the number of resonance lines~\cite{tomas,asymp} and for estimates of partial Franel sums~\cite{PFS}. 
In Section~\ref{sec3} the previous results are easily extended to $\mathcal{N}^{2/k}_n(h)$ and $I_n(2/k)$.

\section{Farey fractions with equal numerators}\label{sec1}

\begin{lem}\label{lema1}
For given positive integer numbers $k>1$, $n$ and $h$, the number of Farey fractions 
in between $1/k$ and $1/(k-1)$ in $F_n$ with numerators equal to $h$ is $\varphi(h)$ when $ n \geq kh-1 $. Furthermore, the number of Farey fractions 
in between $1/k$ and $2/(2k-1)$ in $F_n$ with numerators equal to $h$ is $\varphi(h)/2$, for $h>2$.

\end{lem}
\begin{proof}
Using Theorem~1 in~\cite{PFS} we establish a  map between the Farey fractions in between 0/1 and 1/1 and the Farey fractions in between $1/k$ and $1/(k-1)$:
\begin{equation}
\frac{t}{h} \mapsto \frac{h}{hk-t }\label{map}\ . \nonumber
\end{equation}
Therefore, for every Farey fraction with numerator equal to $h$ in between  $1/k$ and $1/(k-1)$ there is one Farey fraction with denominator equal to $h$ in between $0/1$ and $1/1$ and vice-versa. The number of these fractions is $\varphi(h)$ with a maximum order of $hk-1$. 

Since $2/(2k-1)$ is the image of $1/2$ there are  $\varphi(h)/2$ fractions with numerators equal to $h$ between  $1/k$ and $2/(2k-1)$, for $h>2$.

\end{proof}

\begin{thm} \label{N_n(h)}
For given integers $h$ and $n$, such that $0<h<n$, $\mathcal{N}_n(h)$ is given by
\[
\mathcal{N}_n(h)=n \frac{\varphi(h)}{h} - \varphi(h) - \sum_{d|h}\mu(d)\left\{\frac{n}{d}\right\}   + \delta_{1h}\ ,
\]
where $\mu(x)$ is the M\"obius function, $\{x\}$ represents the fractional part of $x$ and $\delta_{xy}$ is the Kronecker delta symbol. 
\end{thm}
\begin{proof}
For given $n$ and $h$ we are going to count all Farey fractions of the form $h/k$ such that $n\geq k>h>0$, noting that this explicitly excludes the fraction $1/1$ and it needs to be added explicitly via the term $\delta_{1h}$,
\begin{eqnarray}
\mathcal{N}_n(h)&=&\delta_{1h}+\sum_{\substack{k=h+1\\{\rm gcd}(h,k)=1}}^n 1 \ \ \ \  = \delta_{1h}+\sum_{k=h+1}^n\, \sum_{d|{\rm gcd}(h,k)}\mu(d) \nonumber\\
&=& \delta_{1h}+\sum_{k=h+1}^n\sum_{\substack{ d| h\\ d|k }}\mu(d) 
= \delta_{1h}+\sum_{\substack{ d| h }} \sum_{\substack{k=h+1\\d |k}}^n \mu(d) \nonumber\\
&=& \delta_{1h}+\sum_{\substack{ d| h }} \left(  \sum_{\substack{k=1\\d |k}}^n \mu(d) \nonumber -  \sum_{\substack{k=1\\d |k}}^h \mu(d) \nonumber \right)\\
&=& \delta_{1h}+\sum_{d|h}\mu(d)\left( \left\lfloor\frac{n}{d}\right\rfloor -  \left\lfloor\frac{h}{d}\right\rfloor   \right) \ , \nonumber
\end{eqnarray}
where we have used that
\[
\sum_{d\mid n}\mu (d) = \left\{  \begin{array}{ll} 1 \text{ if } n=1\\0 \text{ if } n>1 \end{array}\right.\ {\rm and }\ \
\sum_{\substack{k=1 \\ d|k}}^n 1 = \left\lfloor\frac{n}{d}\right\rfloor\ .
\]
Then, using $\lfloor x \rfloor = x - \left\{x\right\}$ and $\sum_{d|h} \mu(d)/d =\varphi(h)/h$, we have
\[
\mathcal{N}_n(h)=n \frac{\varphi(h)}{h} - \varphi(h) - \sum_{d|h}\mu(d)\left\{\frac{n}{d}\right\}   + \delta_{1h}\ .
\]
\end{proof}

\begin{cor} \label{primepower}
For given integers $h$ and $n$, such that $0<h<n$, 
and $h=p^m$ being a prime power with $m>0$, $\mathcal{N}_n(p^m)$ is given by
\[
\mathcal{N}_n(p^m)=\left\lceil(n-p^m) \left(1- \frac{1}{p} \right)\right\rceil
\ .
\] 
\end{cor}
\begin{proof}
This is derived from the previous Theorem~\ref{N_n(h)}  using that 
\[\varphi(p^m)=p^m-p^{(m-1)}\ {\rm and} \ \  
\sum_{d|p^m}\mu(d)\left\{\frac{n}{d}\right\}=-\left\{\frac{n}{p}\right\}\ ,
\] 
for $p$ being a prime number, yielding
\[
\mathcal{N}_n(p^m)=(n-p^m) \left(1- \frac{1}{p} \right) +\left\{\frac{n}{p}\right\}  
=n -p^m+p^{(m-1)} - \left\lfloor\frac{n}{p}\right\rfloor 
\ .
\]

\end{proof}

\begin{cor}\label{corN} For given integers $h$, $n$ and $p$, such that $0<h<n$ and $p \geq 0$
\[
\mathcal{N}_{n+ph}(h)=\mathcal{N}_{n}(h) + p\varphi(h)\ .
\]
\end{cor}
\begin{proof}
This is derived from the previous Theorem~\ref{N_n(h)} and realizing that
\[
\sum_{d|h}\mu(d)\left\{\frac{n}{d}\right\}   = \sum_{d|h}\mu(d)\left\{\frac{n+ph}{d}\right\}\ ,
\]
for any $p\geq 0$.

\end{proof}

\begin{cor}\label{corheven}
For given integers $h$ and $n$, such that $0<h<n/2$,
\[
\mathcal{N}_{n}(2h)=\left\{  
\begin{array}{ll} 
 \mathcal{N}_{n}(h) - \varphi(h) \text{ if h is even } \\\displaystyle
\frac{1}{2}\left(\mathcal{N}_{n}(h) - \varphi(h) -\delta_{1h}
 +\sum_{d|h}\mu(d)w(n/d)\right) \text{ if h is odd}  \end{array}\right. \ ,
\]
where  $w(x)=\lfloor x \rfloor\mod 2$.
\end{cor}
\begin{proof}
For the case that $h$ is even we assume it is expressed as $h=2^mp$, with $m>0$. Starting from Theorem~\ref{N_n(h)},
\begin{eqnarray}
\mathcal{N}_n(2h) &=& n \frac{\varphi(2h)}{2h} - \varphi(2h) - \sum_{d|2h}\mu(d)\left\{\frac{n}{d}\right\} \nonumber \\
 &=& n \frac{\varphi(h)}{h} - 2\varphi(h) - \sum_{d|h}\mu(d)\left\{\frac{n}{d}\right\} 
 - \sum_{d|p}\mu(2^{(m+1)}d)\left\{\frac{n}{2^{m+1}d}\right\} \nonumber \\
 &=& \mathcal{N}_{n}(h) - \varphi(h)\ ,\nonumber
\end{eqnarray}
where we have used that $\varphi(2h)=2\varphi(h)$ for even $h$ and that $\mu(2^{(m+1)}d)=0$ for $m>0$.

For the case that $h$ is odd $\varphi(2h)=\varphi(h)$ and for $d|h$ we have $\mu(2d)=-\mu(d)$, hence
\begin{eqnarray}
\mathcal{N}_n(2h) &=& n \frac{\varphi(h)}{2h} - \varphi(h) - \sum_{d|2h}\mu(d)\left\{\frac{n}{d}\right\} \nonumber \\
 &=& n \frac{\varphi(h)}{2h} - \varphi(h) - \sum_{d|h}\mu(d)\left\{\frac{n}{d}\right\} 
 - \sum_{d|h}\mu(2d)\left\{\frac{n}{2d}\right\} \nonumber \\
 &=& \frac{1}{2}\left(\mathcal{N}_{n}(h) -\delta_{1h} - \varphi(h)
 +\sum_{d|h}\mu(d)\left(2\left\{\frac{n}{2d}\right\} - \left\{\frac{n}{d}\right\} \right)\right)\ ,\nonumber
\end{eqnarray}
where we only need to introduce  $w(x)=2\{x/2\} -\{x\}=\lfloor x \rfloor\mod 2$.

\end{proof}

As an illustration of the above results let us inspect $F_5$ and $F_8$,
\begin{eqnarray}
F_5&=&\left\{ 
\frac{1}{5}, 
\frac{1}{4}, 
\frac{1}{3}, 
\frac{2}{5}, 
\frac{1}{2}, 
\frac{\color{blue}3}{5}, 
\frac{2}{3}, 
\frac{\color{blue}3}{4}, 
\frac{4}{5}, 
\frac{1}{1}\right\}\ , \nonumber\\
F_8&=&\left\{ 
\frac{1}{8},  
\frac{1}{7}, 
\frac{1}{6}, 
\frac{1}{5}, 
\frac{1}{4}, 
\frac{2}{7}, 
\frac{1}{3}, 
\frac{\color{blue}3}{8}, 
\frac{2}{5}, 
\frac{\color{blue}3}{7}, 
\frac{1}{2}, 
\frac{4}{7}, 
\frac{\color{blue}3}{5}, 
\frac{5}{8}, 
\frac{2}{3}, 
\frac{5}{7}, 
\frac{\color{blue}3}{4}, 
\frac{4}{5},
\frac{5}{6}, 
\frac{6}{7}, 
\frac{7}{8}, 
\frac{1}{1}\right\} \ .\nonumber
\end{eqnarray}
The sequence $F_5$ has two fractions with numerators equal to 3 and indeed, from Theorem~\ref{N_n(h)},
\[
\mathcal{N}_5(3)=5 \frac{\varphi(3)}{3} - \varphi(3) -\mu(3)\left\{\frac{5}{3}\right\}=2   \ ,
\]
as $\varphi(3)=2$ and $\mu(3)=-1$. $F_8$ has four    fractions with numerators equal to 3 and indeed, from Corollary~\ref{corN},
\[
\mathcal{N}_{5+1\cdot3}(3)= \mathcal{N}_{5}(3) + 1\cdot\varphi(3)  =4 \ .
\]
$\mathcal{N}_8(2)=3$ and $\mathcal{N}_8(4)=2$ and according to Corollary~\ref{corheven} 
\[
\mathcal{N}_{8}(2\cdot2)= \mathcal{N}_{8}(2) - \varphi(2) \ .
\]
$\mathcal{N}_8(6)=1$, $\mathcal{N}_8(3)=4$, $w(8)=w(8/3)=0$ and according to Corollary~\ref{corheven} 
\[
\mathcal{N}_{8}(2\cdot3)= \frac{1}{2}(\mathcal{N}_{8}(3) - \varphi(3) + \mu(1)w(8)+ \mu(3)w(8/3) ) \ .
\]

\begin{cor}\label{sumN}
For given integers $m$ and $n$  such that $0<m  < n$ the numer of Farey fractions in $F_n$ with numerators below or equal to $m$ is given by 
\[
\sum_{h=1}^m\mathcal{N}_n(h)=1 + n \sum_{h=1}^m\frac{\varphi(h)}{h} - \Phi(m) - \sum_{h=1}^m\sum_{d|h}\mu(d)\left\{\frac{n}{d}\right\}   \ ,
\]
where $\Phi(n)=\sum_{j=1}^n \varphi(j)$ is the Totient summatory function. 
\end{cor}
\begin{proof}
This is directly derived from Theorem~\ref{N_n(h)} and it is given for comparison with expression~(\ref{ansumN}), from Proposition 1.29 in~\cite{andrey}.

\end{proof}

\begin{cor}\label{corN1k}
For given integers $h$, $n$ and $k$ such that $0<h  < n$ and $0<k< n/h$ 
\[
\mathcal{N}^{1/k}_n(h)=n \frac{\varphi(h)}{h} - k\varphi(h) - \sum_{d|h}\mu(d)\left\{\frac{n}{d}\right\}   \ .
\]
\end{cor}
\begin{proof}
\[
\mathcal{N}_n^{1/k}(h)=  \mathcal{N}_n(h)-\delta_{1h}-(k-1)\varphi(h)\ {\rm for}\ k< n/h\ ,
\]
where we have used that the number of Farey fractions in between $1/k$ and $1/(k-1)$ with numerators equal to $h$ is $\varphi(h)$ when $ 1<k<  n/h $, see Lemma~\ref{lema1}. The term $\delta_{1h}$ needs to be subtracted as $F_n^{1/k}$ does not include the fraction $1/k$ by definition.

\end{proof}

The total number of Farey fractions in $F_n$ is given by $\sum_{i=1}^n\varphi(i)$ and the sum of the numerators of Farey fractions in $F_n$ is known to be $\left(1+\sum_{i=1}^ni\varphi(i)\right)/2$, hence the following equalities can be established and validated with Theorem~\ref{N_n(h)},
\begin{eqnarray}
\sum_{h=1}^n \mathcal{N}_n(h) &=& \sum_{i=1}^n \varphi(i)\ , \nonumber\\
\sum_{h=1}^n h\mathcal{N}_n(h) &=& \frac{1}{2}\left(1+\sum_{i=1}^ni\varphi(i)\right)\ .\nonumber
\end{eqnarray}

\section{Rank of unit fractions in $F_n$}\label{sec2}

\begin{thm}\label{Ink}
For given integers $k$ and $n$, such that $0<k\leq n$ 
\[
I_n(1/k)= n\sum_{j=1}^{\lfloor n/k\rfloor}\frac{\varphi(j)}{j} - k\Phi(\lfloor n/k\rfloor) - \sum_{j=1}^{\lfloor n/k\rfloor}\sum_{d|j}\mu(d)\left\{\frac{n}{d} \right\}\ ,
\]
where $\Phi(n)=\sum_{j=1}^n \varphi(j)$ is the Totient summatory function. 
\end{thm}

\begin{proof}
This is simply obtained by adding the number of fractions with numerators equal to $j$, $\mathcal{N}_n^{1/k}(j)$, from $j=1$ up to $j=\lfloor n/k\rfloor$, using Corollary~\ref{corN1k},
\[
I_n(1/k)=\sum_{j=1}^{\lfloor n/k\rfloor} \mathcal{N}_n^{1/k}(j)\ .
\]
\end{proof}
Theorem~\ref{Ink} is a generalization of Theorem~3 in~\cite{PFS}.
When $n$ is a multiple of all integers between 1 and $\lfloor n/ k \rfloor$, the fractional part $\{n/d\}$ is equal to zero and Theorem~\ref{Ink} takes the form of Theorem~3 in~\cite{PFS}.
Theorem~\ref{Ink} allows to establish remarkable equalities between the ranks of different unit fractions, as shown in~Corollary~\ref{cor2}.  

\begin{cor}\label{cor2}
Given $n=ck$ and $n'=ck'$ with $k'=k+p\;{\rm lcm}(1,2,...,c)$, being ${\rm lcm}$ the least common multiple function, for any integers $c>0$ and $p \geq 0$,
\begin{eqnarray}
I_{ck'}(1/k') &=& I_{ck}(1/k) +(k'-k)\left(c\sum_{j=1}^{c}\frac{\varphi(j)}{j} - \Phi( c)  \right)\ , \nonumber \\
I_{ck'}(1/k') - I_{ck}(1/k) &=& I_{c(k'-k)}(1/(k'-k))\ ,\nonumber \\
\frac{I_{c(k'-k)}(1/(k'-k))}{k'-k} &=& c\sum_{j=1}^{c}\frac{\varphi(j)}{j} - \Phi(c)\ . \nonumber
\end{eqnarray}
\end{cor}
\begin{proof}
The first identity is derived from Theorem~\ref{Ink} by realizing that
\[
\sum_{j=1}^{c}\sum_{d|j}\mu(d)\left\{\frac{n}{d} \right\} = \sum_{j=1}^{c}\sum_{d|j}\mu(d)\left\{\frac{n+p\;{\rm lcm}(1,2,...,c)}{d} \right\}  
\]
for any integers $n>0$ and $p\geq 0$. The other two identities are easily derived and are shown for their remarkable shape.

\end{proof}


\section{Rank of  fractions of the form $2/k$ in $F_n$}\label{sec3}
\begin{cor}\label{corN2k}
For given integers $h$, $n$ and $k$ such that $0<h  < n$ and $0<k< 2n/h-1$ and $k$ being  odd  
\[
\mathcal{N}^{2/k}_n(h)=n \frac{\varphi(h)}{h}-  \frac{k}{2}\varphi(h) - \sum_{d|h}\mu(d)\left\{\frac{n}{d}\right\}   \ .
\]
\end{cor}
\begin{proof}
By virtue of Theorem~\ref{N_n(h)} and Lemma~\ref{lema1}
\[
\mathcal{N}^{2/(2k-1)}_n(h)=n \frac{\varphi(h)}{h}- \left(k- \frac{1}{2}\right)\varphi(h) - \sum_{d|h}\mu(d)\left\{\frac{n}{d}\right\}   \ .
\]
By redefining $(2k-1)$ as $k$ with $k$ being an odd integer fulfilling that  $-1<k< 2n/h-1$,
\[
\mathcal{N}^{2/k}_n(h)=n \frac{\varphi(h)}{h}-  \frac{k}{2}\varphi(h) - \sum_{d|h}\mu(d)\left\{\frac{n}{d}\right\}   \ .
\]
\end{proof}

\begin{cor}For $k$ being an odd integer fulfilling $n\geq k > 2$, the rank of $2/k$ in $F_n$ is given by
\[
I_n(2/k)= n\sum_{j=1}^{\lfloor 2n/k\rfloor}\frac{\varphi(j)}{j} - \frac{k}{2}\Phi(\lfloor 2n/k\rfloor) - \sum_{j=1}^{\lfloor 2n/k\rfloor}\sum_{d|j}\mu(d)\left\{\frac{n}{d} \right\}\ .
\]
Furthermore, with $n=ck/2$, $n'=ck'/2$, $k'=k+p\;{\rm lcm}(1,2,...,c)$ for any even integer $c>1$ and any integer $p\geq 0$
\[
I_{n'}(2/k') = I_{n}(2/k) +\frac{k'-k}{2}\left(c\sum_{j=1}^{c}\frac{\varphi(j)}{j} - \Phi( c)  \right) \ .
\]
\end{cor}
\begin{proof}
This follows from Corollary~\ref{corN2k} with the same steps as in the proofs of Theorem~\ref{Ink} and Corollary~\ref{cor2}.

\end{proof}

\section*{Acknowledgements}
Thanks to Daniel Khoshnoudirad and Andrey O. Matveev for proofreading the manuscript and making highly valuable suggestions.

\end{document}